\documentclass[10pt,reqno]{amsart}

\usepackage[a4paper,asymmetric]{geometry}
\usepackage{latexsym, amsmath, amssymb, esint}
\usepackage{bm}
\usepackage{graphics,epsfig,graphicx}
\usepackage{color, xcolor}
\usepackage{tikz}
\usetikzlibrary{patterns}
\usepackage[english]{babel}
\usepackage [latin1]{inputenc}

\newtheorem{thm}{Theorem}[section]

\newtheorem{lemma}[thm]{Lemma}
\newtheorem{prop}[thm]{Proposition}

\newtheorem{rem}[thm]{Remark}

\def\R{\mathbb{R}}

\newcommand{\ep}{\varepsilon}

\def\norma#1#2{\|#1\|_{\lower 4pt \hbox{$\scriptstyle #2$}}}

\def\e{\varepsilon}

\newcommand{\ratio}{\mathcal{G}}
\newcommand{\ratioo}{\mathcal{G}_0}
\newcommand{\uu}{\textsf{v}}

\begin{document}
        
\title[Quantitative isoperimetric inequality in the plane with barycentric distance]{On the quantitative isoperimetric inequality in the plane with the barycentric distance}
                
\author{Chiara Bianchini}
\author{Gisella Croce} 
\author{Antoine Henrot} 

\address{\emph{C. Bianchini:} (corresponding author)  Dipartimento di Matematica ed Informatica
``U.~Dini'',
        Universit\`a di Firenze, Viale Morgagni 67/A, 50134 Firenze, Italy.}
        \email{chiara.bianchini@unifi.it}

\address{\emph{G. Croce:} 
Normandie Univ, France; ULH, LMAH, F-76600 Le Havre; FR CNRS 3335, 25 rue
Philippe Lebon, 76600 Le Havre, France.}
\email{gisella.croce@univ-lehavre.fr}

\address{\emph{A. Henrot:} Institut \'Elie Cartan de Lorraine, Université de Lorraine, CNRS, IECL, F-54000 Nancy, France.}
\email{antoine.henrot@univ-lorraine.fr}

\subjclass{(2010) 28A75, 49J45, 49J53, 49Q10, 49Q20}
\keywords{Isoperimetric inequality, quantitative isoperimetric inequality,
isoperimetric deficit, 
barycentric asymmetry, shape derivative, optimality conditions.}


\begin{abstract}
In this paper we study a quantitative isoperimetric inequality in the plane, related to the isoperimetric deficit $\delta$ and the barycentric asymmetry $\lambda_0$.  
Our aim is to prove that there exists an absolute constant $C$ such that for every planar (convex or compact and connected) set $\Omega$ it holds: $\lambda_0^2(\Omega) \leq C\; \delta(\Omega)$. 
This
generalizes some results  obtained by B. Fuglede 1993 in \cite{Fu93Geometriae}.
For that purpose, we consider a shape optimization problem in which we minimize
 the ratio $\delta(\Omega)/\lambda_0^2(\Omega)$ both in the class of compact connected sets
 and in the class of convex sets.
%
%
\end{abstract}
\maketitle

\section{Introduction}
In the  last thirty years quantitative isoperimetric inequalities have received much attention in the litterature.
The purpose is to quantify the distance that a subset of $\R^n$, $\Omega$, has from an $n$-dimensional ball $B$ of the same measure in terms of the so called isoperimetric deficit $\delta(\Omega)$: \begin{equation}\label{def-delta}
\delta(\Omega)=\frac{P(\Omega)-P(B)}{P(B)}\,
\end{equation}
 (here and later $|\cdot|$ indicates the $n$-dimensional Lebesgue measure).
Several kinds of distances have been proposed to establish 
quantitative isoperimetric inequalities of the type 
\begin{equation}\label{distK}
[dist(\Omega,B)]^k\le C\; \delta(\Omega).
\end{equation}
In 1989, Fuglede \cite{Fu89Transactions} used the Hausdorff distance of a set $\Omega$ from the ball of same volume centered at the barycentre $x^G$ of $\Omega$.
We recall that the barycenter of a set $\Omega$ is  defined as
$\displaystyle
x^G=\frac{1}{|\Omega|}\int_{\Omega} x\,dx\,.
$
 He called it  the \emph{uniform spherical deviation}. 
He proved a series of inequalities for convex sets and \emph{nearly spherical 
sets}, that is, star-shaped sets with respect to their barycentre (which may be taken to be 0) written as
\begin{equation*}
\{y\in \R^n: y=tx(1+u(x)), x\in \mathbb{S}^{n-1}, t\in [0,1]\},
\end{equation*} 
where $u: \mathbb{S}^{n-1}\to \R$ is a positive Lipschitz function,  with 
 $\norma{u}{L^{\infty}}\leq \frac{3}{20n}$ and $\norma{\nabla u}{L^\infty}\leq \frac 12$. 
In \cite{FGP}, the authors proved the same kind of inequalities as (\ref{distK})  for a more general family of sets,  and considering as distance the minimum of 
the Hausdorff distances of a set from all the balls of $\R^n$ with fixed volume.

 [GIRARE la frase]   L. E. Fraenkel proposed a different kind of method to count the distance from a ball  (now called Fraenkel asymmetry $\lambda(\cdot)$), to enlarge the family of sets for which a quantitative isoperimetric inequality hold:
\begin{equation}\label{def-lambda}
\lambda(\Omega)=\inf_{y\in \R^n}\frac{|\Omega \Delta B_y|}{|\Omega|}\,
\end{equation}
where $B_y$ is the ball of center $y$ and such that $|B_y|=|\Omega|$ and $\Delta$ is the symmetric difference of sets.
This distance can be seen as an $L^1$ distance between $\Omega$ and any ball $B_y$, centered at $y\in \R^n$, with the same measure as $\Omega$. 
On the contrary, 
the Hausdorff distance is in some sense an $L^{\infty}$ distance between sets. 
 Many mathematicians studied quantitative isoperimetric inequalities with the Fraenkel asymmetry, establishing sharp inequalities (see for example \cite{HHW},  \cite{H}, \cite{DeG}, \cite{Ca}, \cite{AFN}, \cite{FiMP}, \cite{FuMP}, \cite{CiLe}, \cite{Fuscopreprint}, {\cite{DL}})
 and existence of an optimal set for the related shape optimization problem (see \cite{CiLeexistence} and \cite{BCH_COCV}).

In the spirit of the Fraenkel asymmetry, Fuglede  proposed in \cite{Fu93Geometriae} 
the barycentric asymmetry:
$$
\lambda_0(\Omega)=\frac{|\Omega \Delta B_{x^G}|}{|\Omega|}
$$
where $B_{x^G}$ is the ball centered at the barycentre ${x^G}$ of $\Omega$ and such that $|\Omega|=|B_{x^G}|$.  
Notice that $\lambda_0(\Omega)$ is obviously
easier to compute than $\lambda(\Omega)$.
Fuglede proved that there exists a positive constant (depending only on the dimension $n$) such that
\begin{equation}\label{Fuglede_convex}
\delta(\Omega)\geq C(n)\,\lambda_0^2(\Omega),\quad\text{for any convex subsets $\Omega$ of $\R^n$}.
\end{equation}

In this paper we propose two kinds of generalizations of Fuglede's results \cite{Fu93Geometriae}, in dimension $n=2$, by considering the shape optimization problem $\min \ratioo(\Omega)$, where 
\begin{equation}\label{ratioo}
\displaystyle \ratioo(\Omega)=\frac{\delta(\Omega)}{\lambda_0^2(\Omega)},
\end{equation} 
in two different families of sets.
\begin{enumerate}
\item
We prove that there exists a strictly positive constant $C$ such that  the
inequality $\ratioo(\Omega)\geq C$ holds for compact connected  sets (see Section \ref{section3}). Notice that, as already observed by Fuglede, the connectedness assumption is necessary in order to have a well posed problem (cf. Remark \ref{contro-esempioFuglede}). 
\item  
The existence of a minimizer of the shape functional $\ratioo$ is proved in the class of convex sets (see Section \ref{Section4}).
We also study the regularity  of the optimal set in Section \ref{Section5} and we write different kinds of optimality conditions.
\end{enumerate}

Some observations about the existence and the shape of an optimal set for the minimization of  $\ratioo$ in the plane are here presented.
\begin{itemize}
\item[-]
In this paper we do not prove the existence of an optimal set for the minimization of $\ratioo$
within the class of compact connected  sets (see Remark 
\ref{remarksection2aboutexistence}). However we formulate a conjecture about the shape of a possible optimal set.
\item[-]
We conjecture that the optimal set for $\ratioo$ within the class of convex sets is the stadium $S$, found in \cite{AFN} as the optimal set for the minimization of \begin{equation}\label{ratio}
\displaystyle\ratio{(\Omega)} := \frac{\delta(\Omega)}{\lambda^2(\Omega)};
\end{equation} 
in Section   \ref{Section5} we give some arguments in this direction.
\item[-]
In \cite{BCH_COCV} our aim was to compute the infimum of 
$\ratio{(\Omega)}$. We underline that, since $\lambda(\Omega)\leq
\lambda_0(\Omega)$, the infimum of 
$\ratioo$, would entail  an estimate from below of the infimum of 
$\ratio{(\Omega)}$. We note that for a given set $\Omega$, $\lambda_0(\Omega)$ is much simpler to compute than $\lambda(\Omega)$.
As observed by Fuglede \cite{Fu93Geometriae},  an estimate from below of
the infimum of 
$\ratio{(\Omega)}$ is given
in Lemma 2.1 of \cite{HHW}: one has $\ratio{(\Omega)}\geq
0.02$ for every $\Omega\subset \R^2$; see also \cite{FiMP} for an estimate
in any dimension. In \cite{LiCRAS} the authors found that $\ratio{(\Omega)}\geq
0.0625$ for every $\Omega$ in the plane.
However, better estimates should be possible. In \cite{CiLe}, \cite{BCH_COCV}  the conjectured optimal set for $\ratio{(\Omega)}$ is described:  a kind of {\it mask}
with two axes of symmetry and two optimal disks for the Fraenkel asymmetry, whose boundary is composed by arcs of circle with three different radii.  If this conjecture is proved, it would provide the following
sharp estimate: $\ratio{(\Omega)}\geq 0.3931$.
\end{itemize}
\section{Preliminaries}\label{sec2}
Given a subset $E$ of $\R^2$, we denote by $E^c$ its complementary set and by $co(E)$ its convex hull. $\mathcal{C}$ indicates the set of open convex bounded subsets of $\R^2$. For $\ep>0$ we denote by $E^{\varepsilon}$ the $\varepsilon-$enlargement of $E$, that is, 
$$
E^\ep=\{x\in \R^2: dist(x,E)\leq \varepsilon\}
$$ where $dist$ is the Euclidean distance. 
$Diam(E)$ denotes the diameter of $E$, that is
$$
Diam(E)=\sup_{x,y\in E} |x-y|.
$$

For $\Omega \subset \R^2$ {measurable} and bounded, $\delta(\Omega)$ indicates the isoperimetric deficit as defined in (\ref{def-delta}), where $P(\Omega)$ is the perimeter in the Minkowski sense, that is: 
$$
P(\Omega)= \lim\limits_{\varepsilon\to 0^+}\dfrac{|\Omega^{\varepsilon}| -  |\Omega|}{\varepsilon}\,.
$$
We will explain later in Remark \ref{contro-esempioFuglede} why  this notion of perimeter is adapted to our problem
and why the classical perimeter in the sense of De Giorgi, denoted by $P^{DG}(\Omega)$ in the sequel, is not suitable here (we refer to \cite{AFP, HP} for definitions).

While considering the Fraenkel asymmetry (\ref{def-lambda}) for a set $\Omega$, we refer to a ball $B_x$ as to \emph{an optimal ball} if $|B_x|=\Omega$ and $\lambda(\Omega)=|\Omega\Delta B_x|/|\Omega|$.

We are interested in investigating the shape functionals $\ratioo, \ratio$ defined in (\ref{ratioo}), (\ref{ratio}), respectively.
%
Let $D$ be a fixed closed disk, we indicate by $\mathcal{K}(D)$ the set of all compact connected
subsets of $D$, while $\mathcal{K}$ indicates the class of connected compact subsets of $\R^2$. 
We recall that the {\it Hausdorff distance}
between two sets $K_1$ and $K_2$ in $\mathcal K$ is defined
by 
$$
d_{\mathcal{H}}(K_1,K_2):= \max \left\{ \sup_{x\in K_1} dist(x,K_2), \sup_{x\in
K_2} dist(x,K_1) \right\}. 
$$

We recall the classical Blaschke's  Theorem (cfr. Theorem 2.2.3 in \cite{HP}):  
\begin{thm}\label{Blaschke}
        Let $\{K_n\}$ be a sequence in $\mathcal K(D)$.  Then
there exists a subsequence which converges in the Hausdorff metric to a set
$K\in \mathcal K(D)$.
\end{thm}

\begin{thm}\label{hausdorffconvex}
        Let $\{K_n\}$ be a sequence of compact convex sets converging  in the Hausdorff metric to a set
$K$. Then $K$ is compact and convex.
\end{thm}

We will also use the following semicontinuity result, analogous to the Golab Theorem for the Minkowski perimeter
in the plane, proved by Henrot and Zucco in \cite{HZ}:

\begin{thm}\label{Henrot-Zucco}
 Let $\{K_n\}\subset \R^2$ be a sequence contained in $\mathcal{K}(D)$
converging to a set $K \in \mathcal{K}(D)$ in the Hausdorff metric. 
Then 
\[
P(K)\leq \liminf_{n\to \infty} P(K_n).
\]
\end{thm}

We will also use the following consequences of the Hausdorff convergence of sets (see Proposition 2.2.21 of  \cite{HP}). Here $\chi_K$ denotes the characteristic function of a set $K$.

\begin{prop}\label{prophp} 
Let $K_n, K$ in $\mathcal{K}(D)$. If $K_n\to K$ in the Hausdorff metric, then
\begin{enumerate}
\item
$|K_n
\setminus K|\to 0$
\item
$\chi_{K}\geq \limsup_{n\to \infty}\chi_{K_n}$ a.e.
\item
If $\chi_{K_n}\to \chi$ in $L^1(\Omega)$ (or even weak-star in $(L^1,L^\infty)$), then $\chi\leq \chi_K$.
\end{enumerate}
\end{prop}
We also recall a compactness result about the $L^1$ convergence of sets, that is, the $L^1$ convergence of characteristic functions of sets. See \cite{HP} for the proof.  
\begin{prop}\label{cptL^1}
Let $K_n$ be a sequence of sets contained in an open set with finite measure, such that $P^{DG}(K_n)+|K_n|$ is uniformly bounded. 
Then there exists a set $K$ such that $\chi_{K_n}\to \chi_K$ in $L^1$, up to a subsequence.
\end{prop}

\begin{rem}\label{DeGiorgi-Minkowski}
The perimeters $P^{DG}$ and $P$ of a set satisfy the inequality 
$P^{DG}(K)\leq P(K)$ if $K\subset \R^2$ is a compact connected set, as remarked in \cite{HZ}.
\end{rem}
In   \cite{BCH_COCV} we found the infimum value of $\mathcal{G}$ for sequences converging to a ball. We used the notion of De Giorgi perimeter  to define the isoperimetric deficit. In view of
the above relation between the two notions of perimeters $P^{DG}, P$ one has the following result.
\begin{thm}\label{thmBCH}
Let $\{\Omega_{\varepsilon} \}_{\varepsilon>0}$ be a sequence of planar sets converging
to a ball $B$ in the sense that  
$|B\Delta \Omega_{\varepsilon} |\to 0$ as $\varepsilon\to 0$.
Then
$$
\inf
\left\{
\liminf_{\varepsilon\to 0}\frac{\delta(\Omega_{\varepsilon})}{\lambda^2(\Omega_{\varepsilon})}
\right\} 
\geq 
\frac{\pi}{8(4-\pi)}.
$$
\end{thm}

We will use the following result about the minimization of $\ratio$ 
within the class convex sets, proved in \cite{AFN}.

\begin{thm}\label{AFNthm}
There exists an optimal set for the minimization problem
$\displaystyle
\inf_{K\in\mathcal{C}} \ratio(K)$,
where $\mathcal{C}$ is the family of convex planar sets.
The infimum is attained by an explicitely described stadium $S$ and 
$$
\min\limits_{K\in\mathcal{C}} \ratio(K)=\ratio(S) \approx 0.406.
$$
\end{thm}

\begin{rem}\label{dumbbell}
In the sequel we will use  the set $D$ given by two balls, each one of area $\frac{\pi}{2}$, connected by a segment of length 2, whose direction passes through their centers. We will call it dumbbell. We observe that its Minkowski perimeter counts twice the length of
the segment and therefore
$$
\ratioo(D)=\frac{\delta(D)}{\lambda^2_0(D)}=
\frac{\sqrt{2}-1}{4}+\frac{1}{2\pi}\approx 0.26 <\frac{\delta(S)}{\lambda_0^2(S)}=\ratioo(S)\approx 0.406\,,
$$ 
where $S$ is the stadium of the above theorem.
\end{rem}

In the following we will use nearly spherical sets, studied by Fuglede in \cite{Fu89Transactions}, that is star-shaped sets $E$ parametrized as  
\begin{equation}\label{nearly-spherical}
E=\{y\in \R^2: y=tx(1+u(x)), x\in \mathbb{S}^1, t\in [0,1]\},
\end{equation}
with $u: \mathbb{S}^1\to (0,+\infty)$ a Lipschitz function. 
Moreover, we will assume that the barycenter is at the origin and that $|E|=\pi$.
Let $B$ be the unit ball centered at 0.
Then, it is straightforward to check:
$$
|E\Delta B|=\frac 12 \int_0^{2\pi} |(1+u)^2-1| \,,
$$
$$
\mathcal{H}^1(\partial E)=\int_0^{2\pi} \sqrt{(1+u)^2+|u'|^2} \,,
$$
$$
\int_0^{2\pi} \cos \theta (1+u)^3=0=\int_0^{2\pi} \sin \theta (1+u)^3\,,\,\,\,\,\,\,\int_0^{2\pi} (1+u)^2=2\pi\,.
$$
We will use the  following result by Fuglede (Lemma 2.2 in \cite{Fu89Transactions}): 
\begin{thm}\label{lemma2.2fuglede}
Let $K_n$ be a sequence of convex compact sets of area $\pi$, converging in the Hausdorff metric to the unit ball $B$. Assume $K_n$ to be parametrized in the form
$$
K_n=\{y\in \R^2: y=tx(1+u_n(x)), x\in \mathbb{S}^1, t\in [0,1]\},
$$ 
where $u_n$ is a Lipschitz  function. The following estimate holds:
$$
\|u_n'\|_{L^\infty} \leq 2 \frac{1+\|u_n\|_{L^\infty}}{1-\|u_n\|_{L^\infty}} \,{\|u_n\|^{\frac 12}_{L^\infty}}
\,.
$$
\end{thm}
In \cite{publi_CV}  the following result has been proved. It will be useful in the shape optimization problem within the class of convex sets.
\begin{thm}\label{thm_pb_CV}
Let $m$ be defined by
$$m=\inf_{u\in\mathcal{L}}\frac{ \displaystyle \int_0^{2\pi}[(u')^2-u^2]d\theta}{\displaystyle
\left[\int_0^{2\pi} |u| d\theta\right]^2}
$$
where $\mathcal{L}$ is the space of $H^1(0,2\pi)$ functions satisfying the constraints:
\begin{itemize}
\item[(L1)]
$\displaystyle \int_0^{2\pi} u\,d\theta=0$
\item[(L2)]
$\displaystyle\int_0^{2\pi} u \cos(\theta)\,d\theta=0=\int_0^{2\pi}\sin(\theta)u\,d\theta$
\item[(L3)]
$u(0)=u(2\pi)$.
\end{itemize}
Then $\displaystyle m=\frac{1}{2(4-\pi)}$. 
\end{thm}

\section{Minimization of $\displaystyle \frac{\delta(\Omega)}{\lambda_0^2(\Omega)}$ within the class of compact connected sets}\label{section3}
In this section, we consider compact connected sets $K\in\mathcal{K}$ of positive measure (in order the shape functionals 
$\delta$ and $\lambda_0$ be well-defined).
We are going to prove the following result. 
\begin{thm}\label{thmsection3}
There exists $C>0$ such that  
for any planar compact connected set $K\in\mathcal{K}$ it holds
$$
\lambda_0^2(K)\leq C \delta(K).
$$ 
\end{thm}
In the proof we will use the following simple lemma, which involves the notion of $d_{L^1}(E,F)$ of sets $E,F\in\mathcal{K}$, defined as the $L^1$ distance of their characteristic functions.
 \begin{lemma}\label{lemmapalle}
Let $B_1$ and $B_2$ be two balls such that their area equals $\pi$ and the distance between their centers equals $a\leq 2$.
 Then
$$
d_{L^1}(B_{1},B_{2})=4\arcsin\left(\frac{a}{2}\right)+2a\sqrt{1-\frac{a^2}{4}}.
$$
Moreover, for small values of $a$ it holds $d_{L^1}(B_{1},B_{2})=4a + o(a).$
\end{lemma}
\begin{proof}
Up to a rotation we can assume that $B_1=B_{(0,0)}$ and $B_2=B_{(a,0)}$, where $B_{(a,0)}$ denote the ball of area $\pi$ centered at $(a,0)$, $0\leq a\leq 2$.
Let 
 $\tau=\arcsin(a/2)$. The quantity $d_{L^1}(B_{(0,0)},B_{(a,0)})$ is equal to  4 times the area of the domain $E$ whose boundary is composed by the following three arcs:
\begin{enumerate}
\item
$(a+\cos t, \sin t), t \in (0,\alpha), \alpha=\frac{\pi}{2}+\tau$;
\item
$(\cos t, \sin t), t \in (0,\beta), \beta=\frac{\pi}{2}-\tau$;
\item
$(t,0), t \in (1,1+a)$.
\end{enumerate}
Hence by Green's theorem
$$
d_{L^1}(B_{1},B_{2})=4|E|=\frac 12 \int_0^{\frac{\pi}{2}+\tau}(a+\cos \tau)\cos \tau +\sin^2t\; d\tau-\frac 12 \int_0^{\frac{\pi}{2}-\tau} 1+0\; d\tau=\tau+\frac a2\cos \tau,
$$
and we are done. Moreover as $a$ tends to $0$ it follows $d_{L^1}(B_{1},B_{2})=4a + o(a).$.
\end{proof}

We are now going to prove Theorem \ref{thmsection3}, {showing 
that for a minimizing sequence $K_n$, $\displaystyle \liminf_{n\to \infty} \ratioo(K_n)>0$. We will distinguish the cases where $K_n$ converges to a ball or to a set different from a ball.}
{\begin{proof}[Proof of Theorem \ref{thmsection3}]
Let $K_n$ be a minimizing sequence in $\mathcal{K}$, that is,  
$$
\ratioo(K_n)=\displaystyle \frac{\delta(K_n)}{\lambda_0^2(K_n)} \to \inf_{E\in\mathcal{K}} \ratioo(E)= \inf_{E\in\mathcal{K}} \frac{\delta(E)}{\lambda_0^2(E)}.
$$ 
Without loss of generality, we can assume that all the sets $K_n$ have area $\pi$.
By Theorem \ref{AFNthm}  one has
$$
\ratioo(K_n)\le \ratioo(S)=\ratio(S)\approx 0.406\,,
$$
where $S$ denotes the stadium introduced in Theorem \ref{AFNthm}. 
Since $\lambda_0(E)\leq 2$ for any set $E$, we get 
\begin{equation}\label{perimeterbounded}
P(K_n)\leq 16.6\,.
\end{equation} 
Therefore the sets $K_n$ are all contained in a fixed ball, since they are connected and their perimeter is uniformly bounded. 
Theorem \ref{Blaschke} gives us the existence of a connected compact set $K$ towards which $K_n$ converges in the Hausdorff metric.
\\
There exists a set $\hat{K}$ such that $\chi_{K_n}\to \chi_{\hat{K}}$ in $L^1$ 
and $|\hat{K}|=\pi$, by Proposition \ref{cptL^1} and Remark \ref{DeGiorgi-Minkowski}.
We are going to prove that  $\hat{K}=K$ 
(we note that the only Hausdorff convergence does not guarantee that $|{K}|=\pi$).
\\
By point (3) in Proposition \ref{prophp} applied to $K_n^c$ et $K^c$, we have
\begin{equation}\label{premiereinegalitefonctionscaracteristiques}
\chi_{\hat{K}}\leq \chi_{{K}}\,,
\end{equation}
since $\chi_{K_n}\to \chi_{\hat{K}}$ in $L^1$.
Therefore 
\begin{equation}\label{premiereinegaliteaires}
|K|\geq \pi=|\hat{K}|\,.
\end{equation}
\\
Since $K_n\to K$ in the Hausdorff metric,  $K$ is contained into the $\varepsilon$-enlargement $K_n^{\varepsilon}$ of $K_n$. By the definition of the Minkowski perimeter,
 we have, for every small
$\varepsilon>0$,
\begin{equation}\label{minkowski}
|K|\leq |K_n^{\varepsilon}| = |K_n|+\varepsilon P(K_n)+o(\varepsilon) = \pi+\varepsilon P(K_n)+o(\varepsilon)\,.
\end{equation}
Since
$P(K_n)$ are uniformly bounded, inequality (\ref{minkowski}) 
yields $|K|\leq
\pi$. This inequality and (\ref{premiereinegaliteaires}) imply
$|K|=\pi$. 
We deduce  that 
$K=\hat{K}$ a.e. from  (\ref{premiereinegalitefonctionscaracteristiques}).
\\
Since $K_n\to K$ in $L^1$, as $n\to \infty$, we have $\lambda_0(K_n)\to \lambda_0(K)$. 
Indeed,  by the triangle inequality,
$$
\pi|\lambda_0(K_n)-\lambda_0(K)|\leq  d_{L^1}(K_n,K)+ d_{L^1}(B,B_n)\,.
$$
The first term in the right hand side tends to 0, as $n\to \infty$ by the $L^1$ convergence of $K_n$ to $K$. The second one tends to 0 by Lemma \ref{lemmapalle}, since
$$
\left|x_1^{G_n}-x_1^G\right|\leq \frac{1}{\pi}\int_{K_n\setminus K}|x_1| dx_1dx_2\,,
$$
and the last term tends to 0, since the diameter of $K_n$ is uniformly bounded and $|K_n\setminus K|\to 0$, as $n\to \infty$. The same holds for the second coordinate.
\\
Now, only two possibilities may occur:
either $\lambda_0(K)>0$ or 
$\lambda_0(K)=0$.
\begin{enumerate}
\item
In the first case, since $K_n$ is a minimizing sequence and $\inf_{E\in\mathcal{K}} \frac{\delta(E)}{\lambda_0^2(E)}>0$, one has 
$0<P({K})\leq \liminf\limits_{n\to \infty} P(K_n)$ by Theorem \ref{Henrot-Zucco}. 
Therefore
$$
\displaystyle \liminf_{n\to \infty}\frac{\delta(K_n)}{\lambda_0^2(K_n)}\geq \frac{\delta(K)}{\lambda_0^2(K)}>0\,.
$$
\item
In the second  case  we can assume that
 $\lambda(K_n)=2\varepsilon_n\to 0$, with $\varepsilon_n\to 0$ as $n\to \infty$, since $\lambda_0(K_n)\geq \lambda(K_n)$. 
By Theorem \ref{thmBCH} one has
$$
\delta(K_n)\geq 0.45 \cdot 4 \varepsilon_n^2\,.
$$
We are now going to prove that there exists an absolute positive constant $A$ such that
\begin{equation}\label{difflambda}
|\lambda(K_n)-\lambda_0(K_n)|\leq \frac{4A}{\pi}\varepsilon_n.
\end{equation}
Therefore
\begin{equation}\label{stimaliminf}
\frac{\delta(K_n)}{\lambda_0^2(K_n)}
\geq
\frac{\delta(K_n)}{\left(\lambda(K_n)+\frac{4A}{\pi}\varepsilon_n\right)^2}
\geq \frac{1.8}{\left(2+\frac{4A}{\pi}\right)^2}\,,
\end{equation}
which implies that
$$\displaystyle \liminf_{n\to \infty}\frac{\delta(K_n)}{\lambda_0^2(K_n)}>0\,.
$$
To prove (\ref{difflambda})  it is sufficient to find  a positive constant $A$ such that 
\begin{equation}
\label{distance-centres-a-prouver}
\|G_n- F_n\| \leq A \varepsilon_n 
\end{equation}
where $G_n=(x_1^{G_n},x_2^{G_n})$ is the barycentre of $K_n$ and $F_n$ is the centre of an optimal ball for $\lambda(K_n)$.
Indeed, by the triangle inequality, 
$$
d_{L^1}(K_n,B_{G_n})\leq d_{L^1}(K_n, B_{F_n})+d_{L^1}(B_{G_n}, B_{F_n})\,,
$$
where $B_{F_n}$ is an optimal ball for $K_n$ with respect to the Fraenkel asymmetry.
This inequality together with  (\ref{distance-centres-a-prouver}) and Lemma \ref{lemmapalle} imply (\ref{difflambda}).
\\
We want to prove (\ref{distance-centres-a-prouver}), which will end the proof of this case.
We can always assume that an optimal ball for the Fraenkel asymmetry is centered in 0, that is, $F_n=0$.
We are now going to estimate $\displaystyle x_1^{G_n}=\frac{1}{\pi}\int_{K_n} x_1 dx_1 dx_2$.  
Writing the last integral on $(K_n\setminus B)\cup B\setminus (B\setminus K_n)$ 
and recalling that  $\displaystyle \frac{1}{\pi}\int_{B} x_1 dx_1 dx_2=0$,
we get
$$
 |x_1^{G_n}|=\frac{1}{\pi}\left|\int_{K_n\setminus B} x_1dx_1dx_2 - \int_{B \setminus K_n} x_1dx_1dx_2\right|
 \leq 
 \frac{1}{\pi}\int_{K_n\setminus B} \left|x_1\right|dx_1dx_2 + \frac{1}{\pi} \int_{B \setminus K_n} \left|x_1\right|dx_1dx_2\,.
 $$
We observe that $diam(K_n\setminus B)\leq diam(K_n)$. Now,  $diam(K_n)\to 2$. Indeed
by definition of Hausdorff convergence, $K_n\subset B^{\varepsilon}$ and $B\subset K_n^{\varepsilon}$. Therefore
$diam(K_n)\leq 2 +2\varepsilon$ and $2\leq diam(K_n)+2\varepsilon$.
Using  this property to estimate the first of the last two terms, we get
$$
 |x_1^{G_n}|\leq
 \frac{2 \varepsilon_n}{\pi}+ \frac{\varepsilon_n}{\pi} ={\left(\frac{3}{\pi}\right)}{\varepsilon_n}\,,
 $$
 where we have used that
 $|K_n\setminus B|=|B\setminus K_n|=\varepsilon_n$, since
 $\lambda(K_n)=2\varepsilon_n$.
The same estimate can be obtained for $|x_2^{G_n}|$. Therefore 
\begin{equation}
\label{A}
\|G_n-F_n\|=\|G_n-0\|\leq A\varepsilon_n=\frac{3\sqrt{2}}{\pi}{\varepsilon_n}
\end{equation} and (\ref{distance-centres-a-prouver}) is proved.
\end{enumerate}
\end{proof}}
\begin{rem}\label{remarksection2aboutexistence}
We conjecture that the infimum of $\ratioo$ in the class of the connected sets is attained by the dumbbell described in Remark \ref{dumbbell}. 

In the case where the minimizing sequence $K_n$ converges to the ball, we get  the following
 estimate from below of $\displaystyle \liminf\limits_{n\to \infty} \ratioo(K_n)$ (see (\ref{stimaliminf}) with $A=\frac{3\sqrt{2}}{\pi}$ given in (\ref{A})):
 $$
\liminf\limits_{n\to \infty} \ratioo(K_n) \geq 0.13\,.
 $$
Notice that this estimate is lower than the value of 
$\ratioo$ computed on the dumbbell (see Remark \ref{dumbbell}); this is  the reason why the existence of an optimal set for this problem is still an open problem. 
We remark that the technique introduced in \cite{BCH_COCV}, cannot be applied to the functional $\ratioo$ in the class $\mathcal{K}$ of connected compact sets in order to exclude sequences converging to a ball.
\end{rem}


\begin{rem}\label{contro-esempioFuglede} As mentioned by Fuglede (we refer to \cite{Fuscopreprint}), the assumption that $\Omega$ is connected cannot be removed. Indeed one can construct the following sequence of non-connected sets $\Omega_n$, given by the union of 
the disk centered in $(2,0)$, of radius $R_n=1-\frac{1}{n}$, and the disk
centered in
$\left(-\frac{2(n-1)^2}{2n-1},0\right)$,
 of radius
$r_n=\sqrt{\frac{2n-1}{n^2}}$\,.
It is easy to check that $|\Omega_n|=\pi$, the barycentre of $\Omega_n$ is the origin, $\delta(\Omega_n)=R_n+r_n-1\to 0$ as $n\to \infty$
and $\lambda_0(\Omega_n)=2$. Thus $\lim_{n\to\infty} \ratioo(\Omega_n)=0$.

This example shows  why the classical De Giorgi perimeter is not suitable for the barycentric asymmetry $\lambda_0$. Indeed,  
the set $\tilde{\Omega}_n$
obtained by connecting the above two balls by a long segment
has the same De Giorgi perimeter as the perimeter of $\Omega_n$, since the De Giorgi perimeter of the long segment is null.
Thus $\lim_{n\to\infty} \ratioo(\tilde\Omega_n)=0$.
On the contrary, for the Minkowski perimeter, 
$\delta(\Omega_n)\to +\infty$, since the length of the long segment counts twice.
\end{rem}


\begin{rem}
The notion of Minkowski perimeter has a central role in the second part of the above proof, in inequality (\ref{minkowski}), to prove that $|K|=\pi$.
\end{rem}


\section{Minimisation of $\displaystyle \frac{\delta(\Omega)}{\lambda_0^2(\Omega)}$ within the class of compact convex sets}\label{Section4}
In this section we prove the following result :
\begin{thm}
There exists an optimal set for $\displaystyle \inf\limits_{\Omega\in\mathcal{C}}\ratioo(\Omega)$.
\end{thm}
We recall that  $\mathcal{C}$ is the family of convex planar sets.
\begin{proof}
Let $K_n$ be a minimizing sequence of convex compact sets so that
$$
\liminf\ratioo(K_n)=\inf_{\Omega\subset\mathcal{C}}\ratioo(\Omega).
$$
It follows a uniform bound on $\frac{\delta(K_n)}{\lambda_0^2(K_n)}$ and hence using the definition of $\lambda_0$ we have that $\delta(K_n)$ is uniformly bounded. Therefore the sets $K_n$ are all contained in a fixed ball, since they are convex and they perimeter is uniformly bounded. 

Theorems \ref{Blaschke} and \ref{hausdorffconvex} entail the existence of a convex compact set towards which $K_n$ converges in the Hausdorff metric.
Now, as in the proof of the previous theorem, two possibilities may occur:
\begin{enumerate}
\item
$K_n$ converges to a ball $B$ in the Hausdorff metric;
\item
$K_n$ converges to a set $K$ different from a ball  in the Hausdorff metric.
\end{enumerate}
In the next theorem we are going to analyse the first case, proving  that 
$\liminf \ratioo >\ratioo(S)=0.406$ 
 where $S$ is the stadium of Theorem \ref{AFNthm}. This means that a minimizing sequence cannot converge to a ball. Therefore the only possibility for a minimizing sequence is the second one. In this case we can prove that $K$ is a minimizer by using an analogous argument to that of the proof of case (2) of Theorem \ref{thmsection3}.
\end{proof}


\begin{thm}\label{proposition042}
Let $K_n$ be a sequence of convex compact sets converging to a ball in the Hausdorff metric. Then 
$\liminf\limits_{n\to \infty}\ratioo(K_n) \geq 0.41$.
\end{thm}
Let $E$ be a nearly spherical set with barycenter in 0, and assume $E$ to be written in polar coordinates with respect to 0 as in (\ref{nearly-spherical}).
Hence if $E$ is a convex set, then the functional $\ratioo(E)$  can be written as a function $J(u)$ of the parameter $u$ and, according to the computations in Section \ref{sec2}, we obtain:
\begin{equation}\label{def-J}
J(u)=\frac{\pi}{2} \frac{\displaystyle\int_0^{2\pi}\left[\sqrt{(1+u)^2+u'(\theta)^2}-1\right]d\theta}{\left[\frac 12 \displaystyle\int_0^{2\pi} |(1+u)^2-1|d\theta\right]^2}.
\end{equation}
We are interested in minimizing the functional $J(u)$ with respect to $u\in \mathcal{NL}$, that is, $u\in H^1(0,2\pi)$, satisfying the constraints of fixed area and barycentre in $0$:
\begin{itemize}
\item [(NL1)]
$\displaystyle\frac{1}{2\pi}\int_0^{2\pi} (1+u)^2d\theta=1$;
\item [(NL2)]
$\displaystyle\int_0^{2\pi}\cos(\theta)[1+u(\theta)]^3d\theta=0=\int_0^{2\pi}\sin(\theta)[1+u(\theta)]^3d\theta$;
\item [(NL3)]
$u(0)=u(2\pi)$.
\end{itemize}
This leads to a complicated problem in the calculus of variations, thus, the strategy consists in replacing this problem by a simpler one which can be seen as a sort of linearization. In order to do this, we define
\begin{equation}\label{opepl}
m=\inf_{u\in\mathcal{L}}\frac{ \displaystyle \int_0^{2\pi}[(u')^2-u^2]d\theta}{\displaystyle
\left[\int_0^{2\pi} |u| d\theta\right]^2}
\end{equation}
where $\mathcal{L}$ is the space of $H^1(0,2\pi)$ functions satisfying the constraints:
\begin{itemize}
\item[(L1)]
$\displaystyle \int_0^{2\pi} u\,d\theta=0$;
\item[(L2)]
$\displaystyle\int_0^{2\pi} u \cos(\theta)\,d\theta=0=\int_0^{2\pi}\sin(\theta)u\,d\theta$;
\item[(L3)]
$u(0)=u(2\pi)$.
\end{itemize}

\begin{prop}\label{prop_ultima}
Let $0<\e<1/24$. Consider $J:\mathcal{NL}\to \R$ defined in (\ref{def-J}) and let
\begin{equation}\label{opep}
m_\e:=\inf \{J(u), \|u\|_{L^\infty}=\ep,\;u\in\mathcal{NL}\}.
\end{equation}
It holds
\begin{equation}\label{conclusion}
\liminf_{\e \to 0} m_\e \geq \frac{\pi}{4}m.
\end{equation}
\end{prop}

\begin{proof}
The proof follows the following steps:
\begin{enumerate}
\item we estimate the infimum value fo the  optimization problem \eqref{opep} by considering an auxiliary problem 
(problem \eqref{opepbis}), whose infimum yields a smaller value; 
\item we prove that the auxiliary problem has a minimizer $u_\e$;
\item we prove that  $v_\e=\frac{u_\e}{\e}$ 
(which belongs to the unit sphere of $L^\infty$) is bounded in $H^1$ and converges
uniformly to some
function $v_0$;
\item  passing to the limit as $\e\to 0$, we prove that $v_0$ is a test function for the optimization
problem \eqref{opepl}
whence the desired inequality holds.
\end{enumerate}

In the sequel of the proof $C$ will denote an absolute constant independent of $\varepsilon$.

{\it Step 1.} 
Since $\displaystyle \int_0^{2\pi} (2u+u^2)
=0$ by (NL1), 
the minimization of  $J(u)$ is equivalent to the minimization 
(with the same constraints) of
$$
J_1(u)=\frac{\pi}{2}\displaystyle\frac{\displaystyle \int_0^{2\pi}\left[\sqrt{(1+u)^2+u'^2}-1\right]d\theta
- \int_0^{2\pi} (u+u^2/2)d\theta}{\left[\frac 12 \displaystyle\int_0^{2\pi} |(1+u)^2-1|d\theta\right]^2}\,.
$$
We are going to estimate the numerator of $J_1$; more precisely we are going to prove that
\begin{equation}\label{numJ}
\int_0^{2\pi}\left[\sqrt{1+2u+u^2+{u'}^2}-1 
-(u+u^2/2)\right]d\theta \ge \frac{1}{2}\int_0^{2\pi}\left[{u'}^2 -u^2 -\frac{1}{4} {u'}^4 - u{u'}^2\right]d\theta
+C\e^3,
\end{equation}
for some constant $C$.
Assume that $\e<1/24$ and 
\begin{equation}\label{stimaFuglede}
\|u'\|_{L^\infty}\leq
3\sqrt{\e}
\end{equation}
(this is possible by the estimate 
$\|u\|_{L^\infty}\leq {\e}$ and Theorem \ref{lemma2.2fuglede}). 
We first observe that for $|\rho|\leq \frac{1}{2}$,
\begin{equation}\label{I1}
\sqrt{1+\rho}\ \geq 1+\frac{\rho}{2}
- \frac{\rho^2}{8} +
\frac{\rho^3}{16} -\frac{\rho^4}{8}\,.
\end{equation}
By \eqref{stimaFuglede}, one has the estimate $|2u+u^2+{u'}^2|\leq 2\e+\e^2+9\e\leq 12\e \leq \frac 12$.
We can apply \eqref{I1} to $2u+u^2+{u'}^2$ to infer
$$
\sqrt{(1+u)^2+{u'}^2} -1 \geq u +\frac{1}{2}{u'}^2 -\frac{1}{8} {u'}^4
-\frac{1}{2} u{u'}^2 +C\e^3\,.
$$
Therefore (\ref{numJ}) holds.

We are going to estimate the denominator of $J_1$. 
Since $|2u+u^2|\leq (2+\e)|u|$ one has that
$$
{\displaystyle \left[\frac 12\int_0^{2\pi} |(1+u)^2-1|\right]^2\leq 
\left(1+\frac{\e}{2}\right)^2  \left[\int_0^{2\pi} |u|\right]^2\,}.
$$ 
Therefore, under the constraints (NL1), (NL2), (NL3) one has
$
J(u)=J_1(u)\geq J_2(u)
$, where
\begin{equation}\label{defJ3}
J_2(u)=\frac{\pi}{2}\ \displaystyle\frac{\displaystyle \frac{1}{2}\int_0^{2\pi}\left[{u'}^2
-u^2 -\frac{1}{4} {u'}^4 - u{u'}^2\right] +C\e^3}{\displaystyle \left(1+\frac{\e}{2}\right)^2
 \left[\int_0^{2\pi} |u|\right]^2}\,.
\end{equation}
Defining
\begin{equation}\label{opepbis}
m'_\e:=\inf \{J_2(u), \|u\|_{L^\infty}=\ep,\ u \in\mathcal{NL}\}\,,
\end{equation}
we have $m_\e \geq m'_\e$.

\medskip\noindent

{\it Step 2.}
We prove that problem \eqref{opepbis} has a minimizer $u_{\e}$  for every fixed $\e>0$.

Let $u_n^{\e}$ be a minimizing sequence for $J_2$.
We know that $\|u_n^{\e}\|_{L^\infty}=\e$ and $\|(u_n^{\e})'\|_{L^\infty}\leq 3\sqrt{\e}$ for every $n$ (by  (\ref{stimaFuglede})).  
Therefore $u_n^{\e}\to u_{\e}$ weakly in $W^{1,\infty}(0,2\pi)$ and uniformly in $(0,2\pi)$, as $n\to \infty$.

To pass to the limit, as $n\to \infty$, in $J_2(u_n^\e)$, we need to study   the integral in the numerator of $J_2$. We will use a standard argument in the calculus of variations.
Notice that for small $|s|$ and $|\xi|$ (recall that $|s|\leq \e\leq \frac{1}{24}$ and $|\xi|\leq 3\sqrt{\e}\;$), 
the function
$j(s,\xi)={\xi}^2(1-s)
-s^2 -\frac{1}{4} {\xi}^4$
 is convex with respect to $\xi$. 
This gives
$$
j(s,\xi)\geq
j(s,\eta)+\nabla_\xi j(s,\eta)\cdot(\xi-\eta)\,,
$$
where $\nabla_\xi j(s,\xi)=2\xi (1-s)-\xi ^3$.
Therefore
\begin{equation}\label{disuguaglianza_de_giorgi}
\int_0^{2\pi} j(u_n^{\e},(u_n^{\e})')
\geq 
\int_0^{2\pi} j(u_n^{\e},u'_{\e}) 
+ \int_0^{2\pi}
\nabla_\xi j(u_n^{\e}, u'_{\e})\cdot((u_n^{\e})'-u'_{\e}).
\end{equation}
The uniform convergence of $u_n^{\e}$ to $u_{\e}$
implies that $j(u_n^{\e},u'_{\e})$ converges
in $L^1(0,2\pi)$ to
$j(u^{\e},u'_{\e})$, as $n\to \infty$.
Moreover $(u_n^{\e})'-u'_{\e}$  converges weakly to $0$  in
$ L^\infty(0,2\pi)$ and $\norma{\nabla_\xi j(u_n^{\e}, u'_{\e})}{L^{\infty}}$ is bounded uniformly in $n$ and $\varepsilon$.
Passing to the lim inf in (\ref{disuguaglianza_de_giorgi}),
we get
$$
\liminf_{n\to +\infty}
\int_0^{2\pi} j(u_n^{\e},(u_n^{\e})')
\geq 
\int_0^{2\pi} j(u^{\e},u'_{\e}) 
\,,
$$
which gives an estimate for the numerator of $J_2(u)$.
We deduce the existence of  a minimizer $u_{\e}$ for $J_2$.

{\it Step 3.} Let us consider  the renormalising sequence
$
v_\e=\frac{u_\e}{\e}\,.
$ 
We are going to prove some estimates on $v_\e$ which allows to compute the limit of $v_\e$, as $\e\to 0$.
The estimates on $v_\e$ will be established thanks to 
 the test function $w_\e$, defined here below.
 
Let $a_\e=\pi/4-\e \pi/6$ and $b_\e=3\pi/4-\e\pi/6$.
Let $w^\e$ be the function, piecewise affine, $\pi$-periodic, defined  by
$$w^\e(t)=\left\lbrace
\begin{array}{lc}
\e\ \dfrac{t}{a_\e} & t\in [0, a_\e]\,, \\
-2\e \ \dfrac{t-a_\e}{b_\e-a_\e}+\e & t\in [a_\e, b_\e]\,, \\
-\e\ \dfrac{\pi -t}{\pi -b_\e} & t\in [b_\e, \pi]\,. \\
\end{array}
\right.$$
It is easy to see that $w^\e$ satisfies $\|w^\e\|_{L^\infty}=\e$ and condition (NL3).
It also satisfies (NL2), since 
$(1+w^\e)^3$ is $\pi-$periodic and therefore  orthogonal to sine and cosine. 
We are going to  check
(NL1), that is, 
$$
2\int_0^{\pi} w^\e + \int_0^{\pi} {w^\e}^2 =0.
$$
Elementary calculations provide:
$$
\int_0^{\pi} w^\e =\frac 12 a_\e\e -\frac 12 (\pi -b_\e)\e ,\quad \int_0^{\pi}
{w^\e}^2 =\pi \frac{\e^2 }{3}\,.
$$
Therefore (NL1) is satisfied as soon as $a_\e+b_\e-\pi=-\e/3$ which holds true by the definition of $a_\e$ and $b_\e$.
We also remark that 
$$
\int_0^{2\pi} {{w^\e}'}^2  \leq C\e^2\,,\,\,\,\,\,\,\,\,\,
\int_0^{2\pi} {{w^\e}}^2  \leq C\e^2\,,\,\,\,\,\,\,\,\,\,
\int_0^{2\pi} {{w^\e}'}^4  \leq C\e^4\,,\,\,\,\,\,\,\,\,\,
\int_0^{2\pi} {w^\e{w^\e}'}^2  \leq C\e^3
$$
and
$$
{\int_0^{2\pi} |w^\e| = \e \pi}\geq \frac{\e}{2}\,.
$$
These estimates imply that $J_2(w^\e)\leq C$ for every $\e$. 
Therefore
$J_2(u_{\e})\leq J_2(w^\e)\leq C$ 
which yields to
$$
\int_0^{2\pi} {(u_{\e})'}^2
\leq \int_0^{2\pi} u_{\e}^2+\frac{1}{4}{(u_{\e})'}^4+u_{\e}{(u_{\e})'}^2\; d\theta+C\e^3
+
C\left[\int_0^{2\pi} |u_{\e}|\right]^2\,.
$$
We deduce that
$\displaystyle \int_0^{2\pi} {(u_{\e})'}^2 \leq C\e^2\,.$

From the definition of $v_\e$ we have $\|v_\e\|_{L^\infty}=1$, moreover by the above estimate we get
\begin{equation}\label{controlbis}
\int_0^{2\pi} {v'_\e}^2 \leq C,
\end{equation}
and hence the sequence $v_\e$ is bounded in
$H^1(0,2\pi)$ and, as $\e\to 0$,  $v_\e$ converges weakly in $H^1(0,2\pi)$ and uniformly to some function $v_0$, up to subsequences.
Using \eqref{stimaFuglede} we deduce that $\|v'_\e\|_{L^\infty}\leq \frac{3}{\sqrt{\e}}$,
 and then using \eqref{controlbis}, we have
$$
\int_0^{2\pi} {v'_\e}^4 \leq \frac{9}{\e} \int_0^{2\pi} {v'_\e}^2 \leq
\frac{C}{\e}.$$
Moreover
$$
\left|\int_0^{2\pi}v_\e {v'_\e}^2\right| \leq  \int_0^{2\pi} {v'_\e}^2 \leq {C}
$$
by \eqref{controlbis} and $\|v_\e\|_{L^\infty}=1$.
\\
{\it Step 4.}  We now prove that the function $v_0$ found in {\it Step 3} is a test function for the optimization problem
\eqref{opepl}. This will allow us to prove the statement of this Proposition.
\\
We observe that, by (\ref{opepbis}) and the definition of $v_\e$, we have
\begin{equation}\label{defJ3bis}
m'_\e=\displaystyle \frac{\pi}{2}\frac{\displaystyle \frac{1}{2}\int_0^{2\pi}\left[{v_\e'}^2
-v_\e^2 -\frac{\e^2}{4} {v'_\e}^4 - \e v_\e{v'_\e}^2\right] +C\e}{\displaystyle\left(1+\frac{\e}{2}\right)^2
 \left[\int_0^{2\pi} |v_\e|\right]^2}.
\end{equation}
Passing to the limit in \eqref{defJ3bis}, we get
$$
\liminf_{\e\to 0} m'_\e \geq \frac{\pi}{2}\ \displaystyle\frac{\displaystyle \frac{1}{2}\int_0^{2\pi}\left[{v_0'}^2
-v_0^2\right]}{\displaystyle \left[\int_0^{2\pi} |v_0|\right]^2}.$$
On the other hand, passing to the limit in (NL1) and (NL2), we see that $v_0$
satisfies (L1) and (L2)
and therefore is an admissible test function for the optimization problem
\eqref{opepl}. 
For example, (NL1) is equivalent to $\displaystyle \int_0^{2\pi} (u_\ep^2+2u_\ep)=0$. This implies that 
$
\displaystyle
0\leq \int_0^{2\pi} u_{\ep}^2 = -2  \int_0^{2\pi} u_{\ep}\leq C\e^2
$
by the estimate $\|u_\e\|_{L^{\infty}}\leq \e$. Therefore, by the definition of $v_\e$, one has $\displaystyle \int_0^{2\pi} v_\e\to 0$, as $\ep\to 0$, which gives $\displaystyle\int_0^{2\pi}v_0=0$.

The inequality
$\displaystyle 
\liminf_{\e\to 0} m_\e \geq \liminf_{\e\to 0} m'_\e \geq \frac{\pi}{4}m$ follows.
\end{proof}
We can now prove Theorem \ref{proposition042}.
\begin{proof}
As we already observed, we can write $K_n$ in polar coordinates (see (\ref{polar_coord})); this implies that
$$
\liminf_{n\to \infty}\ratioo(K_n)=
\liminf_{\e \to 0} m_\e 
\geq \frac{\pi}{4}m
$$
where $m$, $m_\e$ are defined by (\ref{opepl}), (\ref{opep}), respectively. .
By Proposition \ref{prop_ultima} 
 it is sufficient to prove the estimate  
\begin{equation}\label{aimODE}
\frac{\pi}{4} m > 0.41\,.
\end{equation}
Notice that by Theorem \ref{thm_pb_CV}, it holds $\displaystyle m=\frac{1}{2(4-\pi)}$.
This implies
that $\displaystyle \frac{\pi}{4} m \approx 0.457$, and the desired estimate holds.
\end{proof}
\begin{rem} 
Although Fuglede \cite{Fu89Transactions} was interested  in the uniform spherical deviation, that is, the Hausdorff distance of a set $E$ from the ball of same measure centered at the barycenter of $E$, one can easily
deduce from his results the inequality
$
\delta(E)\geq C(n)\lambda_0^2(E)
$
for nearly spherical sets (see Theorem 3.1 in \cite{Fuscopreprint}), where $C(n)$ is a constant depending on the dimension. In particular the following estimate can be proved in the plane:
$$
\delta(E)\geq \frac{1}{16}\lambda_0(E)^2\,. 
$$
However this estimate is not sufficient to exclude sequences converging to the ball.

Our first attempt to prove Theorem \ref{proposition042} was the following.
For the denominator of $J$ one has
$$
\left[\frac 12 \displaystyle\int_0^{2\pi} |(1+u(\theta))^2-1|d\theta\right]^2\leq
\left[\frac 12 \displaystyle\int_0^{2\pi} |u(\theta)|(2+\varepsilon)d\theta\right]^2=
\left(1+\frac{\varepsilon}{2}\right)^2\norma{u}{L^1}^2\,.
$$
For the numerator
$$
\int_0^{2\pi}\left(\sqrt{(1+u)^2+u'(\theta)^2}d\theta-1\right)d\theta
 $$
$$
\geq
\int_0^{2\pi} \left[u + \frac{u^2+(u')^2}{2}-\frac{[4 u^2+4u^3+4uu'+u^4+u'^4+2u^2u'^2]^2}{8}+\frac{8u^3}{16}\right]d\theta
$$
$$
\geq
\int_0^{2\pi}
\left[
\frac{(u')^2}{2}-\frac{u^2}{2}\right]
d\theta+o(2)
\geq 
\frac{c-1}{2} \int_0^{2\pi}u^2 d\theta+o(2)
$$
where $c=4$ is such that
$$
\int_0^{2\pi}{(u')^2}d\theta
\geq 
c \int_0^{2\pi}u^2d\theta\,.
$$
Therefore
$$
\liminf_{\varepsilon\to 0} J(u)\geq 
\frac {3\pi}{4}  \displaystyle \frac{\displaystyle\int_0^{2\pi}u^2d\theta}{\displaystyle\left[\int_0^{2\pi}|u|d\theta\right]^2}\geq \frac 38=0.375
$$
by H\"older inequality. 
Again, this estimate is not sufficient to exclude sequences converging to the ball. 
\end{rem}

\section{On the regularity and on the shape of the optimal convex set}\label{Section5}
In this section we prove that an optimal set for the minimization of $\ratioo$ within the class of planar convex sets has $C^{1,1}$ boundary. About its shape, we conjecture that the stadium $S$ which minimizes $\ratio$ (see Theorem \ref{AFNthm}) also minimizes the functional $\ratioo$. In particular we show that within the class of stadia, $S$ is the only one satisfying the optimality conditions, presented in Theorem \ref{ocPrinceton}.
Unfortunately we are not able to prove that $S$ is an optimizer within the class all convex sets.

The proof of the regularity of the optimal set makes the use of the first order optimality
condition in the same spirit of \cite{LNP}.

 Let us first recall how to write these optimality conditions, in
the case of convexity constraint, when representing the boundary of the convex set with the \emph{gauge function}. We recall that the gauge function $\uu$
is  defined through the polar representation of the boundary of the convex set as $\uu(\theta)=1/r(\theta)$, where $(r(\theta)\cos\theta, r(\theta)\sin\theta)$ is a boundary point.  Thus we consider the optimization problem 
\begin{equation}\label{pbthesejimmy}
\min\left\{
j(\uu), \uu\in H^{1}(0,2\pi), \uu''+\uu\geq 0, \int_0^{2\pi} \frac{d\theta}{\uu^2(\theta)}=m_0
\right\}
\end{equation}
(see   Proposition  2.3.3 of \cite{theseLamboley}), where $j(\uu)$ represents an analytic functional related to a shape functional.
\begin{prop}\label{proptheseLamboley}
Assume  that $\uu_0$ solves (\ref{pbthesejimmy}) where $j: H^{1}(0,2\pi)\to \R$ is $C^2$. Then there exist $\xi_0$ non-negative, $\mu \in \R$ such that $\xi_0=0$ on the support of 
$(\uu''_0 + \uu_0)$
and   
$$
dj(\uu_0;\varphi) = <\xi+\xi'',\varphi> - \mu\; dm(\uu_0;\varphi)
$$ 
for every $\varphi \in H^1(0,2\pi)$, where $dj(\uu;\varphi)$ indicates the derivative of $j(\uu)$ along the deformation $\varphi$ and $dm(\uu_0;\varphi)$ is the corresponding derivative of $ m(\uu):=\int_0^{2\pi} \frac{d\theta}{\uu^2(\theta)}$.
\end{prop}

In this section, we prove the following regularity result:
\begin{thm}\label{C11}
A minimizer of $\ratioo$ within the class of convex compact sets of the plane has   $C^{1,1}$ boundary. Moreover the strictly convex parts of the boundary (except the intersection with the circle of the barycentric disk) are of class $C^\infty$.
\end{thm}

For the proof, we first express the optimality condition of Proposition \ref{proptheseLamboley} in our context:
\begin{prop}\label{conditions-optimalite-gauge}
Let $r, \theta$ be the polar coordinates. Let $\uu(\theta)=\frac{1}{r(\theta)}$
be the gauge function used to describe the boundary of a set $E$.
The optimal set satisfies the following condition: there exists
 $\xi \in H^1$ such that $\xi \ge 0$ for every $\theta\in [0,2\pi]$ and $\xi(\theta)=0$ at every $\theta$ corresponding to strictly convex boundary points of $E$  and there exist $\hat{\mu}_0, \hat{\mu}_1, \hat{\mu}_2
\in \R$ such that
\begin{equation}\label{op1L}
-\frac{1}{2\pi \lambda_0^2}\frac{\uu+\uu''}{(\uu^2+\uu'^2)^{\frac 32}} - \frac{2\delta}{\pi
\lambda_0^3}\frac{sign(\uu^2-1)}{\uu^3}=\xi'' + \xi -\frac{\hat{\mu}_0}{\uu^3}-
3\frac{\hat{\mu}_1 \cos \theta +\hat{\mu}_2 \sin\theta}{2 \uu^4}\,.
\end{equation}
\end{prop}
\begin{proof}
We are going to apply Proposition  \ref{proptheseLamboley}. 
To do that, we need  to compute the derivative of the functional $\ratioo$ and those of  the constraints; we will consider the analytic expressions given by $J(v)$, with constraint $v\in \mathcal{NL}$. With abuse of notation we write $J(\uu)$ to consider the functional $J(v)$ rewritten in terms of the gauge function $\uu$, as considered in \cite{theseLamboley}.
\begin{enumerate}
\item
The derivative of $J(\uu)$ along the deformation $\varphi$ is
$$
dJ(\uu;\varphi) =\frac{1}{2\pi \lambda_0^2} dP(\uu;\varphi) - 2\frac{\delta}{\lambda_0^3}\; d\lambda_0(\uu;\varphi),
$$
where the derivative of $P$ is 
$$
dP(\uu;\varphi)=\displaystyle -\int \frac{\uu+\uu''}{(\uu^2+\uu'^2)^{\frac
32}}\varphi\,.
$$
The derivative of  $\pi \lambda_0$ is the derivative of $\displaystyle \int_{\{\uu<1\}}
\chi_{\Omega \setminus B}+ \int_{\{\uu>1\}}\chi_{B \setminus \Omega}$
which gives
$$
\pi\; d\lambda_0(\uu;\varphi)= \displaystyle\frac{1}{2}\int_{\{\uu<1\}} \frac{-2\varphi}{\uu^3} + \frac{1}{2}\int_{\{\uu>1\}}
\frac{2\varphi}{\uu^3}\,.
$$
\item
The constraints (NL) rewritten in terms of $\uu$ are
 \begin{eqnarray*}
 \frac 12 \int \frac{d\theta}{\uu^2(\theta)}&=&\pi,\\ \label{uno}
 \frac  12\int_0^{2\pi} \frac{\cos \theta}{\uu^3(\theta)}&=&0\\ \label{due}
 \frac 12\int_0^{2\pi} \frac{\sin
 	\theta}{\uu^3(\theta)}&=&0.\label{tre}
 \end{eqnarray*}
 
By derivating each left-hand sides of the above conditions, and introducing the corresponding Lagrange multipliers $\mu_0,\mu_1,\mu_2$, we obtain respectively:
$$\displaystyle -\int \frac{\hat{\mu}_0}{\uu^3},\qquad 
-3\int \frac{\hat{\mu}_1 \cos \theta }{2 \uu^4},  \qquad
-3\int \frac{\hat{\mu}_2 \sin\theta}{2 \uu^4}\,,
$$
so that the desired result follows by Proposition \ref{proptheseLamboley}.
\end{enumerate}
\end{proof}

We are now able to prove Theorem \ref{C11}:
\begin{proof}
We are going to use the notations of the above proposition.
Notice that on the strictly convex parts of the boundary, $\xi=0$ and $\uu$ satisfies a second order ordinary differential equation:
\begin{enumerate}
\item 
in the exterior of the unit ball $\uu<1$, and so $\uu''$ is continuous by equation (\ref{pbthesejimmy}). By a classical bootstrap argument  $\uu$ is $C^{\infty}$;
\item 
in the exterior of the unit ball $\uu>1$,  and so $\uu''$ is continuous by equation (\ref{pbthesejimmy}). By a bootstrap argument  $\uu$ is $C^{\infty}$;
\item
on the boundary of the unit ball  $\uu=1$, $\uu''$ is bounded, but not continuous. Thus $\uu$ is $W^{2,\infty}$ there.
\end{enumerate}
This and the above proposition imply that on strictly convex parts on $\partial B$, $\uu$ is $C^{1,1}$.

Now, let us prove that $\Omega$ is $C^1$. If this was not the case, we would have a corner for some $\theta_0$.
This implies that the gauge function is such that $\uu''+\uu$ contains a Dirac mass, with a positive weight at $\theta_0$.
Thus, the $H^1$ function $\xi$ appearing in the optimality condition (\ref{op1L}) also satisfies that the quantity $\xi''+\xi$ contains a Dirac mass at $\theta_0$. Now, since $\xi(\theta_0)=0$ and $\xi\geq 0$, the weight of this Dirac mass must be non-negative, in contradiction with the minus sign appearing in the left-hand side of \eqref{op1L} in front
of $\uu''+\uu$.

We are left with the conjunctions between a strictly convex part of the boundary and a non strictly convex part. For that, it is sufficient to remark that 
any  $C^{2}(\R^+) (C^{\infty}(\R^+))$ function, which is zero at $x=0$, can be extended by 0 on $\R$, getting a  $C^{1,1}(\R)$ function. This ends the proof that an optimal set is $C^{1,1}$. 
\end{proof}

We are going to write the optimality conditions on strictly convex parts in a different way. In particular, this will
give the explicit expression of the Lagrange multipliers  in \eqref{op1L}.
We can assume that all the considered sets have area equal to $\pi$. 

\begin{thm}\label{ocPrinceton}
Let $\Omega$ be an optimal set minimizing the functional $\ratioo$ and assume its barycenter is at the origin. Let $B$ be the unit ball centered at the origin. Let $\partial \Omega^{IN}=\partial \Omega \cap B$, $\partial \Omega^{OUT}=\partial \Omega \cap B^c$, $\partial B^{IN}=\partial B \cap \Omega$, $\partial B^{OUT}=\partial B \cap \Omega^c$. Then at every strictly convex boundary point $(x,y)$ of $\Omega$ the curvature $C(x,y)$ satisfies:
\begin{equation}\label{opcond}
{C(x,y)}=1-3\delta+\frac{4\delta}{2\pi \lambda_0} \left(|\partial B^{OUT}|-|\partial B^{IN}|\right)
\pm \frac{4\delta}{\lambda_0} +\hat{\mu_1} x + \hat{\mu_2} y\,, 
\end{equation} 
($+$ at the exterior of $B$ and $-$ in the interior of $B$)
where
$$
\hat{\mu_1}= \frac{4\delta}{\pi\lambda_0} \left[\int_{\partial B^{OUT}} \cos t dt -  \int_{\partial B^{IN}} \cos t dt \right]\,,
$$
$$
\hat{\mu_2}=\frac{4\delta}{\pi\lambda_0} \left[\int_{\partial B^{OUT}} \sin t dt -  \int_{\partial B^{IN}} \sin t dt \right]\,.
$$
\end{thm}

\begin{proof}
We are going to perform shape variations on the strictly convex parts of $\partial \Omega$.
The proof is divided into several steps.  Let $V$ be a perturbation, that is a smooth map $V:\R^2\to\R^2$. We denote by $V\cdot n$ the scalar product of $V$ with the outer unit normal vector $n$ to $\partial \Omega$.
\begin{enumerate}
\item
Let $\Omega_t=(I+t V)(\Omega)$. Then
$$
|\Omega_t|=\pi +t \int_{\partial \Omega}  V\cdot n +o(t)\,.
$$
\item
The barycenter constraint implies that
 $\displaystyle \int_{\Omega_t} x\;dxdy= 0 +t\int_{\partial \Omega} x V\cdot n +o(t)$.
Since by definition
$\displaystyle 
x_t=\frac{1}{|\Omega_t|}\int_{\Omega_t} x\; dxdy
$,
by the above formulas one has
$$
x_t=\frac{t}{\pi}\int_{\partial \Omega} x\;V\cdot n  +o(t)\,.
$$
A similar formula holds for $y_t$:
$$
y_t=\frac{t}{\pi}\int_{\partial \Omega} y\;V\cdot n +o(t)\,.
$$
\item
Let $B_t=(I+tW)(B)$, where
$$
\displaystyle W(x,y)=(a,b)+\alpha (x,y)\,,
$$
with
$$
(a,b)=\frac{1}{\pi}\left(\int_{\partial\Omega} x V\cdot n, \int_{\partial\Omega} y V\cdot n \right)
\,, \qquad
\alpha = \frac{1}{2\pi} \int_{\partial\Omega} V\cdot n\,.
$$
\item
The difference between  $|\Omega_t \Delta B_t|$ and $|\Omega \Delta B|$ is given by two terms:
$$\displaystyle 
|\Omega_t \Delta B_t|-|\Omega \Delta B|= \pm \; t \int_{\partial B} W\cdot n \pm t \int_{\partial \Omega} V\cdot n\,:
$$
for the first term of the right hand side  $+$ is on  $\partial B^{OUT}$ and 
$-$ is on $\partial B^{IN}$;
for the last term of the right hand side,  $+$ is $\partial \Omega^{OUT}$ and 
$-$ on $\partial \Omega^{IN}$.

In the next part of the proof we will write
$\displaystyle 
|\Omega_t \Delta B_t|=|\Omega \Delta B|+  t R
$.
\item
We have
$$
\lambda_0(\Omega_t)=\frac{|\Omega_t \Delta B_t|}{|\Omega_t|}=\frac{|\Omega \Delta B| + tR}{|\Omega| + t\int_{\partial \Omega} V\cdot n}= \lambda_0(\Omega)\cdot \frac{1+t\frac{R}{|\Omega \Delta B|}}{1+ \frac{t}{\pi}\int_{\partial \Omega} V\cdot n}= \lambda_0(\Omega) + t\left[\frac{R}{\pi}-\frac{\lambda_0}{\pi} \int_{\partial \Omega} V\cdot n\right]\,.
$$
$$
d\lambda_0(\Omega,V)=\frac{1}{\pi} \left[\pm \int_{\partial B} W\cdot n \pm \int_{\partial \Omega} V\cdot n - \lambda_0 \int_{\partial \Omega} V\cdot n   \right]\,,
$$
that is,
$$
d\lambda_0(\Omega,V)=\frac{1}{\pi} \left[ \int_{\partial B^{OUT}} W\cdot n - \int_{\partial B^{IN}} W\cdot n + \int_{\partial \Omega^{OUT}} V\cdot n - \int_{\partial \Omega^{IN}} V\cdot n - \lambda_0 \int_{\partial \Omega} V\cdot n   \right]\,.
$$
\item
If $r_t$ is the radius of the ball having the same area as  $\Omega_t$, then 
$$
r_t=\sqrt{\frac{\pi + t\int_{\partial \Omega} V\cdot n}{\pi}}=1+t\; \frac{\int_{\partial \Omega} V\cdot n}{2\pi}.
$$ 
This gives
$$
\delta(\Omega_t)=\frac{P(\Omega_t)}{2\pi r_t} -1=\frac{P(\Omega_t)}{2\pi + t \int_{\partial \Omega} V\cdot n} -1=\frac{P(\Omega) + t\int_{\partial \Omega} \mathcal{C}\ V\cdot n}{2\pi + t \int_{\partial \Omega} V\cdot n} -1\,,
$$
where $\mathcal{C}$ indicates the curvature of the boundary.
With the same computations as for $\lambda_0$
$$
\delta(\Omega_t)=\delta(\Omega) + t\frac{\int_{\partial \Omega} \mathcal{C}\ V\cdot n}{2\pi} -t\frac{P(\Omega) \int_{\partial \Omega}V\cdot n}{4\pi^2}  
$$
and so
$$
d\delta(\Omega,V)=\int_{\partial \Omega} \left[\frac{\mathcal{C}}{2\pi}-\frac{(\delta +1)2\pi}{4\pi^2} \right]V\cdot n = 
\int_{\partial \Omega} \frac{\mathcal{C}- \delta -1}{2\pi} V\cdot n\,.
$$
\end{enumerate}

The optimality condition for $\ratioo$:
$$
\displaystyle
\frac{d\delta}{\lambda_0^2}-\frac{2\delta}{\lambda_0^3} d\lambda_0=0,
$$
can be written as
$$
\int_{\partial \Omega} (\mathcal{C}-\delta - 1)V\cdot n = \frac{4\delta}{\lambda_0} \left [\int_{\partial B^{OUT}} W\cdot n - \int_{\partial B^{IN}} W\cdot n +\int_{\partial \Omega^{OUT}} V\cdot n - \int_{\partial \Omega^{IN}} V\cdot n 
- \lambda_0 \int_{\partial \Omega} V\cdot n \right] \,.
$$
Now,
$W\cdot n=a\cos\theta + b\sin\theta +\alpha$ (since $(x,y)\cdot n=1$ on $\partial B$), so
$$
W\cdot n =\cos\theta \int_{\partial\Omega} x V\cdot n +\sin\theta \int_{\partial\Omega} y V\cdot n
+\frac{1}{2\pi} \int_{\partial\Omega} V\cdot n
$$
which gives
$$
\mathcal{C}=\delta+1-4\delta+\frac{4\delta}{2\pi \lambda_0} \left(|\partial B^{OUT}|-|\partial B^{IN}|\right)
\pm \frac{4\delta}{\lambda_0} +\hat{\mu_1} x + \hat{\mu_2} y\,. 
$$
\end{proof}

\begin{rem}\label{ultimo_stadio}
If $S$ denotes the stadium of Theorem \ref{AFNthm}, we can prove that $S$ is the only stadium satisfying the optimality conditions (\ref{opcond}).

To see that, let us consider a stadium centered at $0$, 
given by the union of  a rectangle of dimensions $2r \times 2l$ with $r\le 1$ and two half discs of radius $r$. Let $\theta$ be the angle such that
 $r=\sin \theta$. Assuming without loss of generality that  the area is $\pi$, one has 
$$l=\frac{\pi-\pi \sin^2 \theta}{4\sin \theta}\,.$$ 
The perimeter equals $4l+2\pi r=\frac{\pi}{\sin \theta}+\pi \sin \theta$ which implies
$$
\delta(\theta)=\frac{1}{2 \sin \theta}+\frac{\sin \theta}{2}-1\,,
$$
while we can compute $\lambda_0$ as
$$
\lambda_0(\theta)=\frac{2}{\pi}(\pi-2\theta-\sin(2\theta))\,.
$$
On one hand, an optimal stadium is a critical point of the function
$\theta\mapsto \delta(\theta)/\lambda_0^2(\theta)$. This leads to solve the nonlinear equation
\begin{equation}\label{eqop1}
8\sin\theta(1-\sin\theta)^2-\cos\theta(\pi-2\theta-\sin(2\theta))=0.
\end{equation}
It is a simple exercise to prove that this equation has a unique solution, providing the stadium $S$
which corresponds to the value $\theta\sim 0.5750$.
\\
On the other hand, writing condition (\ref{opcond}) for a stadium yields, with the same notations, to the equation
\begin{equation}\label{eqop2}
4\sin\theta -\frac{3}{2}\sin^2(\theta)-\frac{5}{2}+2\frac{ (1-\sin\theta)^2 (\pi-2\theta)}{\pi-2\theta-\sin(2\theta)}=0.
\end{equation}
It is easy to check that  equation \eqref{eqop2} has a unique solution in $(0,\pi/2)$
and this solves equation \eqref{eqop1} too. Therefore there is only one stadium satisfies (\ref{opcond}); this stadium is $S$, since  $\lambda_0$ and $\lambda$ are equal on any symmetric set. 

Now, to prove that the stadium is indeed the minimizer, it certainly requires to prove first that the Lagrange multilplier $\hat{\mu}_1$ and $\hat{\mu}_2$ are both zero,
then a precise analysis should provide the result.
\end{rem}

\subsection*{Acknowledgments} 
 The first author is supported by the GNAMPA group of Istituto Nazionale di Alta Matematica (INdAM).

A part of this work has been realized while the second and the third authors were hosted at IAS at Princeton. The authors thank the Institute for its warm hospitality.

The authors have been supported by the ANR Projects OPTIFORM and SHAPO of the CNRS (Centre National de la Recherche Scientifique) and by the Fir Project 2013 ``Geometrical and qualitative aspects of PDE's'' of MIUR (Italian Ministry of Education).



\begin{thebibliography}{10}


\bibitem{AFN} {A. Alvino, V. Ferone, C. Nitsch},
A sharp isoperimetric inequality in the plane. J. Eur. Math. Soc. (JEMS)
13 (2011), 185-206.

\bibitem{AFP} {L. Ambrosio, N. Fusco, D. Pallara}, Functions of bounded variation and free discontinuity problems. Oxford Mathematical Monographs. The Clarendon Press, Oxford University Press, New York, 2000.
        

\bibitem{BCH_COCV} {C. Bianchini, G. Croce, A. Henrot},
On the quantitative isoperimetric inequality in the plane.
ESAIM: COCV 23 (2017), 517-549.


\bibitem{publi_CV}
{G. Croce, A. Henrot},
A Poincar\'e type inequality with three constraints,
hal-03236684.


\bibitem{Ca} {S. Campi},
Isoperimetric deficit and convex plane sets of maximum translative discrepancy.
Geom. Dedicata 43 (1992), 71-81.


\bibitem{CiLe} {M. Cicalese, G. P. Leonardi}, {A selection principle
for the sharp quantitative isoperimetric inequality}.
Arch. Ration. Mech. Anal. 206 (2012),  617-643.

\bibitem{CiLeexistence} M. Cicalese, G. P. Leonardi,
{Best constants for the isoperimetric inequality in quantitative form}.
J. Eur. Math. Soc. 15 (2013),  1101-1129.



\bibitem{DL}
M. Dambrine, J. Lamboley, Stability in shape optimization with second variation,
Journal of Differential Equations 267, 5 (2019), 3009-3045

\bibitem{DeG}
E. De Giorgi, Sulla propriet\`a isoperimetrica dell'ipersfera, nella classe degli insiemi aventi frontiera orientata di misura finita, (Italian), Atti Accad. Naz. Lincei. Mem. Cl. Sci. Fis. Mat. Nat. Sez. I 8 (1958), 33-44.

\bibitem{FiMP}  
A. Figalli, F. Maggi, A. Pratelli, 
{A mass transportation approach to quantitative isoperimetric inequalities}.
Invent. Math. 182 (2010), 167-211.


\bibitem{Fu89Transactions}{B. Fuglede},
Stability in the isoperimetric problem for convex or nearly spherical domains in $\R^n$. Trans.
Amer. Math. Soc. 314 (1989), 619-638.

\bibitem{Fu93Geometriae}{B. Fuglede},
Lower estimate of the isoperimetric deficit of convex domains in $\R^n$ in terms of asymmetry.
Geom. Dedicata 47 (1993), 41-48.

\bibitem{Fuscopreprint}
N. Fusco, The quantitative isoperimetric inequality and related topics. Bull.  Math. Sciences 5 
(2015), 517-607.

\bibitem{FGP}
N. Fusco, M. S. Gelli, G. Pisante,
On a Bonnesen type inequality involving the spherical deviation.
J.  de Math. Pures et Appliqu\'ees
98 (2012), 616-632.


\bibitem{FuMP} 
{N. Fusco, F. Maggi,  A. Pratelli},
{The sharp quantitative isoperimetric inequality}. 
Ann. of Math.  168
(2008),  941-980.


\bibitem{HHW}
R.R. Hall, R. R., 
W.K. Hayman, W. K.,  
A.W. Weitsman, On asymmetry and capacity.
J. Anal. Math. 56 (1991), 87-123.

\bibitem{H}
R.R. Hall, A quantitative isoperimetric inequality in n-dimensional space. 
J. Reine Angew. Math. 428
(1992), 161-176.

\bibitem{HH}
E.M. Harrell, A. Henrot,
On the maximization of a class of functionals on convex regions and the characterization of the farthest convex set.
Mathematika
56 (2010)  245-265.

\bibitem{HP} A. Henrot, M. Pierre, Shape variation and optimization. A
geometrical analysis. EMS Tracts in Mathematics, 28. European
Mathematical Society (EMS), Zürich, (2018).

\bibitem{HZ}
{A. Henrot, D. Zucco}, 
Optimizing the first Dirichlet
eigenvalue of the Laplacian with an obstacle.
Ann. Sc. Norm. Super. Pisa Cl. Sci. (5) 19 (2019), no. 4, 1535-1559.


\bibitem{theseLamboley}
{J. Lamboley, A. Novruzi}, Polygons as optimal shapes with convexity constraint. SIAM J. Control and Optimization 48 (2009), 3003-3025.

\bibitem{LNP}
J. Lamboley, A. Novruzi, M. Pierre,  Regularity and singularities of optimal convex shapes in the plane. Arch. Ration. Mech. Anal. 205 (2012),  311-343. 


\bibitem{LiCRAS} 
G. Li,  X. Zhao, Z. Ding, R. Jiang,
An analytic proof of the planar quantitative isoperimetric inequality.
C. R. Math. Acad. Sci. Paris 353 (2015), 589-593.

\end{thebibliography}
\end{document}